\documentclass[12pt,a4paper]{article}

\usepackage{titlesec,titletoc}

\usepackage[margin=0.75in,width=6.25in]{geometry}
\usepackage{url}
\usepackage{fancyhdr}
\usepackage{setspace} 
\usepackage{appendix} 
\usepackage{enumerate}
\usepackage{calc}
\usepackage{xspace}
\usepackage{caption}
\usepackage{subcaption}
\usepackage{placeins}  
\usepackage{pdflscape} 
\usepackage{longtable}
\usepackage{xcolor,colortbl,booktabs}
\usepackage{parskip}  
\usepackage{authblk}  
\usepackage{bold-extra} 
\makeatletter
\def\thm@space@setup{%
\thm@preskip=\parskip \thm@postskip=0pt
}
\makeatother

\usepackage{amsmath}
\usepackage{amssymb}
\usepackage{amsthm}
\usepackage{mathtools}
\usepackage{stmaryrd}

\usepackage{algorithm}
\usepackage{algorithmic}
\usepackage{tikz}
\usetikzlibrary{decorations.pathreplacing}
\usepackage{pgfkeys}

\usepackage{graphicx}
\graphicspath{ {./figures/} }

\newcommand{\mcf}{\mathcal}
\newcommand{\mbb}{\mathbb}

\newcommand{\tpose}{^{\!\top\!}}

\renewcommand{\Re}{\mbb{R}}

\newcommand{\eqdef}{:=}

\newcommand{\norm}[1]{\left\| #1 \right\|}
\newcommand{\set}[2]{\left\{ #1\ \left| \ #2 \right. \right\}}

\newcommand{\interior}{\mathop{\operatorname{int}}}
\newcommand{\trace}{\mathop{\operatorname{tr}}}

\newcommand{\innerprod}[2]{\left\langle{#1},{#2}\right\rangle}
\newcommand{\sqtwo}[1]{\left\|{#1}\right\|^2}
\newcommand{\indexset}[1]{\llbracket {#1} \rrbracket}
\newcommand{\mat}{\text{mat}}
\newcommand{\diag}{\text{diag}}
\newcommand{\rank}{\text{rank}}

\newtheorem{thm}{Thoerem}[section]

\theoremstyle{remark}

\usepackage{xcolor,calc}

\definecolor{BenchHighlight}{rgb}{0.8,0.8,0.9}

\title{\LARGE \textbf{A warmstarting technique for general conic optimization in interior point methods}}
\author[1]{Yuwen Chen}  
\author[2]{Paul Goulart}
\author[1]{Colin Jones}  

\affil[1]{School of Engineering, EPFL, Lausanne, Switzerland}
\affil[2]{Department of Engineering Science, University of Oxford,
Oxford, UK}

\date{\today}
\begin{document} 
\maketitle

 \begin{abstract}
We propose a novel warmstarting method for primal-dual interior point methods based on a smoothing operator that generates a starting point on the central path from the previous optimum. Compared to traditional approaches that prioritize minimizing infeasibility residuals, our method focuses on maintaining proximity to the central path. Computation of a smoothing operator is efficient and can be parallelized for conic constraints. We also prove that the residual of the smoothed starting point remains comparable to the one before the smoothing step. The numerical tests show that the proposed warmstarting strategy can reduce iteration numbers and computational time effectively across test problems. 
 \end{abstract}



\section{Introduction}
Conic optimization is a powerful generalization of linear programming that extends its capabilities by incorporating convex cones into the constraint structure~\cite{bv_2004}. While retaining linear equality constraints, conic formulations allow for more expressive inequality constraints through cones such as second-order cones, exponential cones, power cones and positive semidefinite cones~\cite{MOSEK,Clarabel}. This added flexibility makes conic optimization highly versatile, enabling it to effectively model uncertainty and risk in a wide range of applications, including risk measurement of portfolio optimization~\cite{Krokhmal2007}, the safety guarantees in control engineering~\cite{Boyd94} and robust model design in machine learning~\cite{Xu2009}. 

The primal-dual interior point method~\cite{Nesterov1999,Wright97} is a second-order method based on the Newton method and has been widely used for conic optimmization problems. It uses a barrier function to ensure that the iterates remain within the feasible region of a cone where the function penalizes solutions that approach the boundary of it. The self-scaled property~\cite{Nesterov1997} of a barrier function determines whether a cone is symmetric or not, and also affects the way we characterize the central path in the implementation of a primal-dual interior point method. The central path builds up the link between primal and dual variables with the centering parameter $\mu$ via the logarithmically homogeneous self-concordant barrier function. It has a lot of useful quantities that link different orders of derivatives given the cone degree $\nu$ and are detailed in~\cite{Nesterov1997}. The symmetric cones of the self-scaled property enables the use of Jordan algebra and an alternative form of the central path is used in practice~\cite{Nesterov1997,cvxopt}.

Warmstarting is a technique that can significantly reduce the number of iterations required and can save computational time. It is effective for parametric programming~\cite{MPT3} where the structure and dimensions of a problem remain consistent over time, with only slight variations in parameter values. When these parameters are updated, the optimal solution from the previous problem can be reused as the initial point for an optimization solver. Applications of this class include portfolio optimization~\cite{Krokhmal2007}, model predictive control~\cite{Borrelli:MPCbook}, and hyperparameter tuning for machine learning models~\cite{Bishop:MLbook}. First-order methods have no additional requirements over the intial point and warmstarting new problem from the previous optimum is directly applicable and is already supported in many first-order solvers like SCS~\cite{SCS} and OSQP~\cite{osqp}.

Compared to first-order methods, warmstarting interior point methods is not straightforward since interior point methods often struggle with starting points that are too close to the boundary of the feasible region, which can block the search direction and slow the convergence of interior point methods. Research into warmstarting interior point methods has been studied for many years~\cite{Yildirim2002,Gondzio2002,Gondzio2008} although most of them are for the case of linear programming or quadratic programming with linear constraints rather than general conic optimization problems. A tailored warmstarting strategy is also proposed for optimal control problems where the temporal causality imposes constraints on adjacent state and input variables~\cite{Shahzad2011}.

Warmstarting for general conic optimization in interior point methods is a relatively new research direction as conic optimization has become more popular in recent years. Skajaa et al.~\cite{Skajaa2013} proposed a warmstarting approach for homogeneous and self-dual interior point methods by taking the convex combination of the optimal solution of the previous problem and the default coldstarting point. An extension work~\cite{Cay2017} presented how to exploit the idea of convex combination within mixed integer second-order cone programming. However, the ratio of convex combination is chosen empirically without taking into account the residual level of the original optimal solution on the perturbed problem. 

In summary, the core idea for warmstarting an interior point method is to generate a starting point that keeps infeasibility violation small and is close to the central path. 

In the context of parametric programming, where the previous optimum can be leveraged for warmstarting, we proposed a method that meets both of these criteria. Previous warmstarting approaches for conic optimization~\cite{Skajaa2013} primarily focus on reducing feasibility residuals, placing less emphasis on maintaining proximity to the central path. Moreover, these methods are limited to symmetric cones that admit a Jordan algebra structure. On the contrary, our approach prioritizes proximity to the central path over feasibility satisfaction, guaranteeing large step size in the following interior point method while still keeping the feasibility residuals reasonably small. In addition, the proposed warmstarting strategy is also applicable to general cone constraints. 
The main contributions of our approach are outlined below:

\begin{enumerate}
	\item We propose the smoothing operator for general convex cones with valid logarithmically-homogeneous self-concordant barrier functions. The smoothing operator is a proximal operator that is decomposable with respect to each conic constraint, and the computation can be parallelized. 
	\item We propose a warmstarting strategy based on the smoothing operation. We prove that our warmstarting strategy can generate a starting point on the central path of primal-dual interior point methods.
	\item We prove that the residual of the starting point is at the same level as the previous optimum for nonnegative cones, second-order cones and positive semidefinite cones. The new error introduced by the smoothing step is proportional to the norm of constraint matrix and the initial centering parameter $\mu^0$. 
	\item The experiments illustrate that the proposed warmstarting strategy can effectively reduce the number of iterations and computation time of a primal-dual interior point algorithm when the perturbation is not too large. The experiments validate that the reduction ratio of both time and iteration number decrease when the magnitude of perturbation increases. 
\end{enumerate}

\paragraph{Notation}
We denote $\mathcal{I}_{\mathcal{K}}(x)$ as the indicator function of cone $\mathcal{K}$ and $\Pi_{\mathcal{K}}(x)$ as the projection of $x$ onto cone $\mathcal{K}$. The dual cone of $\mathcal{K}$ is denoted as $\mathcal{K}^*$. The negative of dual cone $\mathcal{K}^*$ is called the polar cone denoted as $\mathcal{K}^{\circ}$. We denote $\interior\mathcal{C}$ as the interior set of $\mathcal{C}$. $\indexset{n}$ denotes the index set $\{1, 2, \dots, n\}$. The second-order cone is abbreviated as $\mathcal{K}_{\text{soc}}^n := \{(x_0,x_1) \in \mathbb{R}^n \ | \ x_0^2 - \norm{x_1}^2 \ge 0\}$. $\mathbb{S}^n_+$ denotes the set of positive semidefinite cones. $\diag(x)$ transforms a vector $x \in \mathbb{R}^n$ into a diagonal matrix and $\mat(s)$ transforms a vector $s \in \mathbb{R}^{n \times n}$ into a square matrix.

\section{Background}
We consider the following conic optimization problem with a quadratic objective in this work:
\begin{align}
	\begin{aligned}
		p^* := \min_{x,s} \quad &  \frac{1}{2}x^\top P x + q^\top x \\
		\text {s.t.} \quad & Ax + s = b, \ s \in \mathcal{K} 
	\end{aligned}\label{primal-problem}
\end{align}
where $P \in \mathbb{R}^{n \times n}$ is symmetric positive semidefinite, $A \in \mathbb{R}^{m \times n}, q \in \mathbb{R}^n$ and$b \in \mathbb{R}^m$. The dual problem is 
\begin{align}
	\begin{aligned}
		d^* := \max_{x,z} \quad & -\frac{1}{2}x^\top P x - b^\top z \\
		\text {s.t.} \quad &  Px + A^\top z + q= 0, \ z \in \mathcal{K}^*.
	\end{aligned}\label{dual-problem}
\end{align}
We assume that the primal problem~\eqref{primal-problem} and the dual problem~\eqref{dual-problem} satisfy Slater's condition and the optimal solution $(x^*,s^*,z^*)$ satisfies the \textit{Karush-Kuhn-Tucker} (KKT) condition
\begin{align}
	\begin{aligned}
		Ax^* + s^* = b, \\
		Px^* + A^\top z^* + q= 0, \\ 
		s^* \in \mathcal{K}, z^* \in \mathcal{K}^*,\\
		\langle s^*, z^* \rangle = 0.
	\end{aligned}\label{KKT-condition}
\end{align}

\subsection{Logarithmically-homogeneous self-concordant barrier function}
The interior-point methods deal with conic constraints by imposing penalties over $s \in \interior\mathcal{K}, z \in \interior\mathcal{K}^*$. The function value $f(s)$ should converge to $+\infty$ when any sequence of points in cone $\mathcal{K}$ converges to a boundary point of $\mathcal{K}$ (barrier property). A function is \emph{self-concordant} if it satisfies the barrier property and
\begin{align}
	\begin{aligned}
		|\nabla^3 f(s)[r,r,r]| \le 2\left(\nabla^2 f(s)[r,r]\right)^{3/2}, & \ \forall s \in \interior(\mathcal{K}), r \in \mathbb{R}^d.
	\end{aligned}\label{concordant}
\end{align}
Moreover, a self-concordant function is called a $\nu$-\emph{logarithmically-homogeneous self-concordant barrier} (LHSCB) function for cone $\mathcal{K}$, if it further satisfies 
\begin{align}
	\begin{aligned}
		f(\lambda s) = f(s) - \nu \log(\lambda), & \ \forall s \in \interior(\mathcal{K}), \lambda > 0.
	\end{aligned}\label{homogeneous}
\end{align}
We call $\nu > 0$ the \emph{degree} of $f$. The convex conjugate $f^*$ of function $f$ is defined as
\begin{align}
	f^*(y) := \sup_{s \in \interior(\mcf K)} \{-\langle y,s \rangle - f(s)\}, \label{fundamental-conjugate-barrier-1}
\end{align}
which is also a $\nu$-LHSCB for $\mathcal{K}^*$~\cite{Nesterov1997} and we call $f^*(y)$ the \textit{conjugate barrier}. The gradient $\nabla f^*$ of $f^*$ is the solution of
\begin{align}
	\nabla f^*(y) := - \arg \sup_{s \in \interior(\mcf K)} \{-\langle y,s \rangle - f(s)\}. \label{fundamental-conjugate-gradient}
\end{align} 
There are some key properties of LHSCB $f(s), \forall s \in \interior(\mcf K)$:
\begin{align}
	\begin{aligned}
		\nabla f(\tau s)=\frac{1}{\tau} \nabla f(s), \quad \nabla^2 f (\tau s)=\frac{1}{\tau^2} \nabla^2 f(s), \\
		\left\langle \nabla f(s), s\right\rangle=-\nu,
	\end{aligned}\label{LHSCB-property-1}
\end{align}
and its relation with the conjugate barrier $f^*(y), \forall y \in \interior(\mcf K^*)$:
\begin{align}
	\begin{aligned}	
		-\nabla f(s) \in \interior(\mcf K^*), -\nabla f^*(y) \in \interior(\mcf K),\\
		f^*\left(-\nabla f(s)\right)=-\nu-f(s), \quad f\left(-\nabla f^*(y)\right)=-\nu-{f^*}(y), \\
		\nabla f^*\left(-\nabla f(s)\right)=-s, \quad \nabla f\left(-\nabla f^*(y)\right)=-y.
	\end{aligned}\label{LHSCB-property-2}
\end{align}
Relevant barrier functions for commonly supported cones in conic solvers can be found in Appendix~\ref{appendix:barrier-function}.

\subsection{Central path}
The LHSCB functions are not only used to tackle the conic constraints $\mathcal{K},\mathcal{K}^*$ but also used to smooth the complementarity condition $\innerprod{s}{z}$ within an interior point method. The residual mapping $R(x,s,z)$ is defined as 
\begin{align}
	\begin{aligned}
		R(x,s,z) = 
		\begin{bmatrix}
			r_d\\ r_p
		\end{bmatrix} :=
		\begin{bmatrix}
			\hphantom{+}P  & A\tpose  \\
			-A & 0 \\
		\end{bmatrix}
		\begin{bmatrix}
			x \\ z 
		\end{bmatrix}
		+
		\begin{bmatrix}
			q \\ b
		\end{bmatrix}
		-	
		\begin{bmatrix}
			0\\s
		\end{bmatrix}.
	\end{aligned}\label{residual-map}
\end{align}
The central path is defined as
\begin{align}
	\begin{aligned}
		R(x,s,z) &= \mu R(x^0,s^0,z^0), \\ 
		z &= -\mu \nabla f(s),
	\end{aligned}\label{general-central-path}
\end{align}
where $(x^0,s^0,z^0)$ is the initial point at the beginning of the interior point method and $\mu > 0$ is called the centering parameter of the central path. It is known that $\mu$ is continuously decreasing as an interior point method proceeds, and the limiting point of~\eqref{general-central-path} at $\mu \rightarrow 0^+$ is also the solution of the original KKT system~\eqref{KKT-condition}. For parametric programming where we need to solve a class of problems of the same structure repeatedly with different parameters, finding a warmstarting point $(x^0,s^0,z^0)$ such that the initial residual $R(x^0,s^0,z^0)$ and the initial $\mu$ are smaller than the default cold-start point can reduce the number of iterations and computational time in a new problem.  We can possibly utilize the optimal solution from the last problem when we assume the parameters of the new problem are only slightly different from the last one.

\subsection{Moreau envelope}
The \emph{Moreau envelope} or \emph{Moreau-Yosida regularization} is the foundation of the proximal algorithms~\cite{Parikh2014} and defined as
\begin{align}
\begin{aligned}
	f_{\alpha}(v)=\inf_x \left\{f(x)+\frac{1}{2 \alpha}\|x-v\|_2^2\right\},
\end{aligned}\label{Moreau-envelope}
\end{align}
where $\alpha > 0$ is the regularization term and $f(x)$ is a proper lower semi-continuous convex function. When we choose $\alpha=1$, \eqref{Moreau-envelope} then becomes the well-known proximal operator, i.e. $\text{prox}_f(v) := f_{1}(v)$. The \emph{Moreau decomposition} states that 
\begin{align}
\begin{aligned}
	\text{prox}_f(v) + \text{prox}_{f^*}(-v) = v
\end{aligned}\label{Moreau-decomposition}
\end{align}
when we follow the definition of conjugate function $f^*$ in~\eqref{fundamental-conjugate-gradient}~\cite[Section 2.5]{Parikh2014}. If $f$ is set to be an indicator function of a closed convex cone $\mathcal{K}$, then~\eqref{Moreau-decomposition} reduces to 
$$
v = \Pi_{\mathcal{K}}(v) + \Pi_{\mathcal{K}^{\circ}}(v) = \Pi_{\mathcal{K}}(v) - \Pi_{\mathcal{K}^*}(v) ,
$$ 
which says any point $v \in \mathbb{R}^n$ can be decomposed into two parts that fall into the cone $\mathcal{K}$ and its dual cone $\mathcal{K}^*$ (or its polar cone $\mathcal{K}^{\circ}$) respectively. This is an important property underlying many operator splitting methods~\cite{COSMO,Chen23a,SCS} where the intermediate primal-dual iterate $\{s^k,y^k\}, s^k \in \mathcal{K}, y^k \in \mathcal{K}^{\circ}$ always satisfies the complementarity condition $ \innerprod{s^k}{y^k} = 0$ for any iteration $k$.

\subsection{Smoothing operator}
Conic constraints are addressed differently between 1st-order optimization algorithms and 2nd-order optimization algorithms. Conic feasibility $s \in \mathcal{K}, z \in \mathcal{K}^*$ is guaranteed by projecting intermediate iterates back to conic constraints in 1st-order algorithms. The projection operation can be interpreted as a proximal operator $\Pi_{\mathcal{K}}$ as
\begin{align}
\begin{aligned}
	\Pi_{\mathcal{K}}(c) := \arg \min_s & \ \frac{1}{2} \sqtwo{s - c} + \mathcal{I}_{\mathcal{K}}(s),
\end{aligned}\label{projection-operator}
\end{align}
where $\mathcal{I}_{\mathcal{K}}$ is the indicator function of cone $\mathcal{K}$. The function $\mathcal{I}_{\mathcal{K}}$ is nonsmooth and not differentiable, so only the generalized Hessian information can be exploited for acceleration, as used in semismooth Newton methods~\cite{Ali2017,Li2018}.

Recently, an ADMM-based interior point (ADMM-IPM) method has been proposed~\cite{ABIP}, where the indicator function in the projection operator~\eqref{projection-operator} is replaced by the LHSCB $f(\cdot)$ of cone $\mathcal{K}$,
\begin{align}
\begin{aligned}
	S_{\mathcal{K},\eta}(c) := \arg \min_s & \ f_{\eta}(s) := \frac{1}{2} \sqtwo{s - c}  + \eta f(s),
\end{aligned}\label{smoothing-operator}
\end{align}
where $\eta>0$ is called the smoothing parameter. The barrier function $f(\cdot)$ was originally used in interior point methods for penalizing conic constraints, but it has been shown that the optimal solution of ADMM-IPM is the point on the central path of classical interior point methods, given a fixed centering parameter $\eta > 0$~\cite{ABIP}. 

We call the operator $S_{\mathcal{K},\eta}(\cdot)$ above the \emph{smoothing} operator, since the objective within the minimization problem~\eqref{smoothing-operator} is continuously differentiable. Moreover, problem~\eqref{smoothing-operator} is the minimization over a strictly convex function and the solution of $S_{\mathcal{K},\eta}(c)$ exists and is unique given the value of $c$. 

\section{Warmstarting point on the central path}

Given a starting point $v^0 = (x^0, s^0, z^0)$, the complexity of an interior point method is $O\left( \sqrt{\nu} \frac{\Phi(v^0)}{\epsilon}\right)$ where
$$
\Phi(v) = \max \{\mu(v), \norm{r_p(v)}, \norm{r_d(v)} \},
$$
$\nu$ is the degree of the cone $\mathcal{K}$ and $\epsilon$ is the desired accuracy level. The standard initial point for primal dual interior point methods is $C := (0, e_s, e_z)$, where $e_s = e_z$ is the identity vector defined in Jordan Algebra for symmetric cones and $e_s, e_z$ are constant vectors satisfying $e_z = - \nabla f(e_s), e_s \in \interior{\mathcal{K}}, e_z \in \interior{\mathcal{K}}$ for nonsymmetric cones~\cite{Dahl21}.

To obtain a better worst case complexity, we would need to initialize an interior point method at a point $v^0$ that is better than the cold start point $C$, which can be quantified by
\begin{align}
\begin{aligned}
    \mu (v^0) < \mu (C), \ \|r_p(v^0)\| < \|r_p(C)\|, \ \|r_d(v^0)\| < \|r_d(C)\|.
\end{aligned}\label{eqn:residual-metric}
\end{align}
In addition, the initial point also has to lie in the neighborhood $\mathcal{N}(\beta)$ of the central path
\begin{align}
    \mathcal{N}(\beta) := \{(s, z) \in \mathcal{K} \times \mathcal{K}^* | \nu_i \langle \nabla f (s_i), \nabla f^* (z_i) \rangle^{-1} \ge \beta \mu, i=1,\dots, p \},
\end{align}
where $\beta = 1$ characterizes the central path~\eqref{general-central-path}. It is known that the cold start $C$ satisfies $\mu(C)=1$ and lies on the central path. The ideal warmstarting point $v^0$ should satisfy~\eqref{eqn:residual-metric} that yields a smaller $\Phi (v^0)$, and stay close to the central path, i.e. $\mathcal{N}(\beta)$ for a $\beta \in (0,1]$ as close to $1$ as possible, that can push the initial point away from cone boundary and accept a large step size for interior point update. 

Our new warmstarting algorithm for primal dual interior point methods in conic optimization is summarized in Algorithm~\ref{alg:warmstart}:
\begin{algorithm}
	\caption{Warmstarting algorithm for parametric programming}
	\begin{algorithmic}[1]
		\REQUIRE Input $(x^*, s^*, z^*)$ from the last optimization problem, a given smoothing parameter $\mu^0 > 0$ and a scaling ratio $\lambda > 0$. \\
		\vspace*{2mm}
		\STATE Compute $c := s^* - \lambda z^*$.
		\STATE Compute the smoothed primal variable $s^0 = S_{\mathcal{K},\mu^0}(c)$ by~\eqref{smoothing-operator}.\\
        \STATE Compute the smoothed dual variable $z^0 = \frac{s^0 - (s^*-\lambda z^*)}{\lambda}$.
        \STATE Set $x^0 = x^*$
		\vspace*{2mm}
        \STATE Output $(x^0, s^0, z^0)$
	\end{algorithmic}\label{alg:warmstart}
\end{algorithm}
Note that there is no valid barrier function for equality constraints corresponding to $\mathcal{K} = \{0\}^n$. Hence, we choose the initial point $(s^0,z^0) = (s^*,z^*)$ for primal and dual variables when $\mathcal{K} = \{0\}^n$.

We prove that Algorithm~\ref{alg:warmstart} can generate an initial primal-dual pair $(s^0,z^0)$ on the central path.

\begin{thm}\label{theorem-1}
	Suppose we have a solution $(x^*,s^*,z^*)$ from a given optimization problem, a smoothing parameter $\mu^0$ and a scaling parameter $\lambda$ for the smoothing operator~\eqref{smoothing-operator}. The initial point generated by Algorithm~\ref{alg:warmstart} satisfies $(s^0, z^0) \in \interior\mathcal{K} \times \interior\mathcal{K}^*$ by setting $f(\cdot)$ to the $\nu$-LHSCB function of $\mathcal{K}$. If we also set $x^0 = x^*$, then the initial point $(x^0,s^0,z^0)$ is on the new central path parametrized by
	\begin{align}
		\begin{aligned}
			R(x,s,z) = \mu r^0, \\
			z = -\mu \nabla f(x),
		\end{aligned}\label{new-central-path}
	\end{align}
	where $r^0 = \frac{\lambda}{\mu^0} R(x^0,s^0,z^0)$, and we have $\innerprod{s^0}{z^0} = \frac{\nu \mu^0}{\lambda}$, i.e. $(s^0, z^0)$ is on the central path with parameter $\mu = \frac{\mu^0}{\lambda}$.
\end{thm}
\begin{proof}
	The problem~\eqref{smoothing-operator} is strongly convex and therefore we can obtain a unique solution $s^0 = S_{\mathcal{K},\mu^0}(s^*- \lambda z^*)$ of it when we set $c=s^*- \lambda z^*$ with $\lambda > 0$. The constrained domain of LHSCB function $f(\cdot)$ ensures $s^0 \in \interior\mathcal{K}$. Considering the optimality condition of the problem~\eqref{smoothing-operator} with $\mu = \mu^0$,
	\begin{align*}
		s^0 - (s^*- \lambda z^*) + \mu^0 \nabla f(s^0) = 0,
	\end{align*}
	which implies
	\begin{align*}
		\lambda z^0 = s^0 - (s^*- \lambda z^*) = -\mu^0 \nabla f(s^0) \in \interior\mathcal{K}^*
	\end{align*}
	due to~\eqref{LHSCB-property-2} and 
	\begin{align*}
		\innerprod{s^0}{z^0} = -\frac{\mu^0}{\lambda}\innerprod{s^0}{\nabla f(s^0)} = \frac{\nu \mu^0}{\lambda}.
	\end{align*}
	due to~\eqref{LHSCB-property-1}.
\end{proof}

\section{Choice of the initial centering parameter $\mu^0$}
The centering parameter $\mu$ is used to characterize the central path and the limiting point $0$ corresponds to the optimal solution of the optimization problem. The initial point $(x^0,s^0,z^0)$ is determined by the choice of initial smoothing parameter $\mu^0$, which affects both the complementarity value $\innerprod{s^0}{z^0}$ but also the linear residual $R(x^0,s^0,z^0)$. Both have to be kept small for fast convergence of interior point methods. We propose the following strategy for selecting smoothing parameter $\mu^0$ in Theorem~\ref{theorem-1}, which is based on the information from the previous optimal solution $x^*,s^*,z^*$. For nonnegative cones, second-order cones and PSD cones, we prove that our warmstarting guarantees the linear residual $R(x^0,s^0,z^0)$ stays in the same precision level as the residual $R(x^*,s^*,z^*)$ before the smoothing step.
\begin{thm}\label{theorem-2}
	Suppose $x^*,s^*,z^*$ is the optimal solution from the previous problem. We choose the value $\mu^0$ to be the same magnitude of the residual $R(x^*,s^*,z^*)$ in the new problem, e.g. $\mu^0 = \norm{R(x^*,s^*,z^*)}_\infty$. For Algorithm~\ref{alg:warmstart}, we have the following results:
	\begin{itemize}
		\item For $s^*, z^* \in \mathbb{R}^n$, we assume at least one of $s^*_i,z^*_i$ is nonzero $\forall i \in \indexset{n}$ and set $\lambda = 1$ in Algorithm~\ref{alg:warmstart}. The residuals of the initial point $(x^0, s^0,z^0)$ is within the magnitude of $O(\mu^0)$ for optimization problems with linear constraints.
		\item For $S^*, Z^* \in \mathcal{K}_{\succeq}^n$, we assume $\rank(S^*)+\rank(Z^*) = n$ and set $\lambda = 1$ in Algorithm~\ref{alg:warmstart}. The residuals of the initial point $(x^0, s^0,z^0)$ is within the magnitude of $O(\mu^0)$ for optimization problems with PSD constraints.
        \item For a second-order cone constraint where $s^*:=(s^*_0,s^*_1), z^*:=(z^*_0,z^*_1) \in \mathbb{R} \times \mathbb{R}^{n-1}$ and $s^*,z^* \in \mathcal{K}_{\mathrm{soc}}^n$, we assume $s_0^*,z_0^* > 0$ and set $\lambda = s_0^*/z_0^*$. The value $\mu^0$ is set to $\mu^0 = \lambda \norm{R(x^*,s^*,z^*)}_\infty$. The residuals of the initial point $(x^0, s^0,z^0)$ is within the magnitude of $O(\mu^0)$ for second-order cone constraints.
	\end{itemize}
\end{thm}
\begin{proof}
	Given the settings in Theorem~\ref{theorem-1}, we have $s^0-z^0 = s^*-z^*$ for $c = s^*-z^*$, which implies the change is equal for both primal and dual variables, i.e. $\Delta s = s^0 - s^*= z^0 - z^* = \Delta z$. The complementarity slackness says $\innerprod{s^*}{z^*} = 0$. 
	
	\begin{enumerate}
		\item We first prove $O(\mu^0)$ deviation after the smoothing step for nonnegative cones. Assume at least one of $s^*_i,z^*_i$ is nonzero $\forall i \in \indexset{n}$, then $c_i$ is either positive when $s_i^* > 0$ or negative when $z^*_i > 0$. We also assume $\mu_0 \ll |c_i|,\forall i \in \indexset{n}$.
		
		For $c_i > 0$, we have $s^*_i =c_i > 0$ and the change $\Delta s_i$ becomes 
		\begin{align}
			\Delta s_i = s^0_i - s^*_i = \frac{\sqrt{c_i^2 + 4 \mu^0} - c_i}{2} = \frac{2\mu^0}{\sqrt{c_i^2 + 4 \mu^0} + c_i} \le \min \{\frac{\mu^0}{c^0},1\},
		\end{align}
		is of magnitude $O(\mu^0)$ for $\mu_0 \ll |c_i|$. For $c_i < 0$, we have $s^*_i = 0$ and the change 
		\begin{align}
			\Delta s_i = s_i^0 = \frac{\sqrt{c_i^2 + 4 \mu^0} + c_i}{2} = \frac{2\mu^0}{\sqrt{c_i^2 + 4 \mu^0} - c_i} \le \min \{\frac{\mu^0}{|c^0|},1\},
		\end{align}
		is also of the magnitude $O(\mu^0)$. Then, the change of residual due to the smoothing operator~\eqref{smoothing-operator} is
		\begin{align*}
			\norm{R(x^0,s^0,z^0) - R(x^*,s^*,z^*)}_\infty = 
			\norm{\begin{matrix}
					A^\top \Delta z\\ \Delta s
			\end{matrix}}_\infty \le (\norm{A}_\infty+1)\norm{\Delta s}_\infty,
		\end{align*}
		which implies
		\begin{align*}
			\begin{aligned}
				\norm{R(x^0,s^0,z^0)}_\infty & \le \norm{R(x^*,s^*,z^*)}_\infty + \norm{R(x^0,s^0,z^0) - R(x^*,s^*,z^*)}_\infty\\
				& \le \norm{R(x^*,s^*,z^*)}_\infty + (\norm{A}_\infty+1)\norm{\Delta s}_\infty.
			\end{aligned}
		\end{align*}
		Setting $\mu^0 = \norm{R(x^*,s^*,z^*)}_\infty$ in the inequality above yields
		\begin{align*}
			\begin{aligned}
				\norm{R(x^0,s^0,z^0)}_\infty \le \left[1 + \frac{\norm{A}_\infty+1}{\min_{i \in \indexset{n}}\{|c_i|\}}\right]\mu^0 .
			\end{aligned}
		\end{align*}
		Moreover, the duality gap is related to the complementarity quantity $\innerprod{s^0}{z^0}/\nu = \mu^0$, which is also dependent on the choice of $\mu^0$ like the linear residual $R(x^0,s^0,z^0)$.
		
		\item For positive semidefinite cones, we assume $\rank(S^*)+\rank(Z^*) = n$. Since the previous optimal solution $S^*,Z^*$ satisfies $S^* Z^*=0$, we can find the eigenvalue decomposition such that 
		\begin{align*}
			S^* = Q^\top \begin{bmatrix}
				D_1 & \\
				& 0
			\end{bmatrix}Q, \quad Z^* = Q^\top \begin{bmatrix}
			0 & \\
			& D_2
		\end{bmatrix}Q,
		\end{align*}
		where $Q$ is the matrix for eigenvectors and $D_1,D_2$ are diagonal matrices corresponding to positive eigenvalues. The complementarity condition implies
		\begin{align*}
			C = Q^\top \begin{bmatrix}
				D_1 & \\
				& -D_2
			\end{bmatrix}Q.
		\end{align*}
		Combining it with~\eqref{psd-solution}, we find the change of variables $\Delta S, \Delta Z$ are
		\begin{align*}
			\Delta S = \Delta Z = Q^\top \Delta D Q,
		\end{align*}
		where $\Delta D$ is
		\begin{align}
			(\Delta D)_{ii} := \left\{\begin{matrix}
				\frac{\sqrt{d_i^2 + 4\mu^0}-d_i}{2} = \frac{2\mu^0}{\sqrt{d_i^2 + 4\mu^0}+d_i}, & d_i > 0\\
				\frac{\sqrt{d_i^2 + 4\mu^0}+d_i}{2} = \frac{2\mu^0}{\sqrt{d_i^2 + 4\mu^0}-d_i} & d_i < 0
			\end{matrix}\right., \quad \forall i \in \indexset{n}.
		\end{align}
	Hence, the change of primal variable is bounded by
	\begin{align*}
		\norm{\Delta s}_\infty = \norm{\Delta S}_{\max} \le \norm{\Delta S} \le \norm{\Delta D} \le \min \left\{\frac{\mu^0}{\min_{i \in \indexset{n}}\{|d_i|\}}, 1\right\},
	\end{align*}
	and the change of dual variable
	\begin{align*}
	\begin{aligned}
		\norm{A^\top \Delta z}_\infty & = 
	\norm{
		\begin{matrix}
			\vdots \\
			\trace\left(A_i \Delta Z\right)\\
			\vdots
		\end{matrix}
	}_\infty \le 
	\norm{
		\begin{matrix}
			\vdots \\
			\trace|A_i|\norm{\Delta Z}\\
			\vdots
		\end{matrix}
	}_\infty \\
	& \le 
	\max_{i \in \indexset{m}} \{\trace|A_i|\} \norm{\Delta Z} \le \frac{\max_{i \in \indexset{m}} \{\trace|A_i|\}}{\min_{i \in \indexset{n}}\{|d_i|\}}\mu^0.
	\end{aligned}
	\end{align*}
	Then, the change of residual due to the smoothing operator~\eqref{smoothing-operator} is
	\begin{align*}
		\norm{R(x^0,s^0,z^0) - R(x^*,s^*,z^*)}_\infty = 
		\norm{\begin{matrix}
				A^\top \Delta z\\ \Delta s
		\end{matrix}}_\infty \le \frac{\max_{i \in \indexset{m}} \{\trace|A_i|\}+1}{\min_{i \in \indexset{n}}\{|d_i|\}}\mu^0,
	\end{align*}
	which implies
	\begin{align*}
		\begin{aligned}
			\norm{R(x^0,s^0,z^0)}_\infty & \le \norm{R(x^*,s^*,z^*)}_\infty + \norm{\begin{matrix}
					A^\top \Delta z\\ \Delta s
			\end{matrix}}_\infty\\
			& \le \left[1 + \frac{\max_{i \in \indexset{m}} \{\trace|A_i|\}+1}{\min_{i \in \indexset{n}}\{|d_i|\}}\right]\mu^0 .
		\end{aligned}
	\end{align*}
	The complementarity quantity is $S^0 Z^0 = S^0(S^0 - C) = S^0 \cdot \mu (S^0)^{-1}= \mu$.
	
	\item For $s^*, z^* \in \mathcal{K}_{\mathrm{soc}}^n$ from the previous optimization problem, we have $\langle s^*,z^* \rangle = s_0^* z_0^* + \langle s_1^*,z_1^* \rangle = 0$, due to the complementarity slackness. Also, we have $\langle s^*,z^* \rangle = s_0^* z_0^* + \langle s_1^*,z_1^* \rangle \ge s_0^* z_0^* - \|s_1^* \|\cdot \|z_1^*\| \ge \|s_1^* \|\cdot (z_0^* - \|z_1^*\|) \ge 0$ due to  $s^*, z^* \in \mathcal{K}_{\mathrm{soc}}^n$. The active constraints imply $s_0^* = \|s_1^*\|$, $z_0^* = \|z_1^*\|$ and $\langle\frac{ s_1^*}{s_0^*},\frac{z_1^*}{z_0^*} \rangle = -1$ ($s_1^*, z_1^*$ have the opposite direction). We can derive both $s^*$ and $z^*$ are on the boundary of $\mathcal{K}_{\mathrm{soc}}^n$ and $s_1^* = -\lambda z_1^*$ with $\lambda = \frac{s_0^*}{z_0^*}$. Then, we have
	\begin{align*}
		c = \begin{bmatrix}
			0 \\
			-2\lambda z_1^*
		\end{bmatrix}.
	\end{align*}
	According to the analysis in Section~\ref{subsec-socp}, the smoothed value $s^0,z^0$ are
	\begin{align*}
		s^0 = \begin{bmatrix}
			\sqrt{\mu^0+\norm{c_1}^2/4}\\
			\frac{c_1}{2}
		\end{bmatrix} = 
		\begin{bmatrix}
			\sqrt{\mu^0+\norm{s_1^*}^2}\\
			s_1^*
		\end{bmatrix}, \quad 
		z^0 = \frac{s^0 - c}{\lambda} = 
		\begin{bmatrix}
			\sqrt{\mu^0/\lambda^2+\norm{z_1^*}^2}\\
			z_1^*
		\end{bmatrix},
	\end{align*}
	which result in the complementarity quantity $\innerprod{s^0}{z^0} = \mu^0/\lambda$ and $s^0,z^0$ are on the nonlinear trajectory $z = -\mu \nabla f(x)$ at $\mu = \mu^0/\lambda$. The change of variables are, given $s^*_0 = \norm{s^*_1}$,
	\begin{align*}
		\Delta s = s^0 - s^* = 
		\begin{bmatrix}
			\sqrt{\mu^0+{s_0^*}^2} - s^*_0\\
			0
		\end{bmatrix} = 		
		\begin{bmatrix}
			\frac{\mu^0}{\sqrt{\mu^0+{s_0^*}^2} + s^*_0}\\
			0
		\end{bmatrix}, \quad
		\Delta z = z^0 - z^* = \frac{s^0 - s^*}{\lambda} = \Delta s/\lambda.
	\end{align*}
	Hence, the change of residual due to the smoothing operator $s^0 = S_{\mathcal{K},\mu^0}(s^*-\lambda z^*)$ is
	\begin{align*}
		\norm{R(x^0,s^0,z^0) - R(x^*,s^*,z^*)}_\infty = 
		\norm{\begin{matrix}
				A^\top \Delta z\\ \Delta s
		\end{matrix}}_\infty \le (\norm{A}_\infty/\lambda+1)\norm{\Delta s}_\infty,
	\end{align*}
	which implies
	\begin{align*}
		\begin{aligned}
			\norm{R(x^0,s^0,z^0)}_\infty & \le \norm{R(x^*,s^*,z^*)}_\infty + (\norm{A}_\infty/\lambda+1)\norm{\Delta s}_\infty\\
			& \le \left[\frac{1}{\lambda} + \frac{\norm{A}_\infty/\lambda+1}{2s^*_0}\right]\mu^0 = \left[1 + \frac{\norm{A}_\infty+\lambda}{2s^*_0}\right]\norm{R(x^*,s^*,z^*)}_\infty.
		\end{aligned}
	\end{align*}
	\end{enumerate}
\end{proof}

\section{Efficient computation for smoothing operators}
In this section, we will show that the smoothing operator~\eqref{smoothing-operator} is computationally efficient for several commonly supported cones in state-of-the-arts solvers~\cite{MOSEK, Clarabel, CuClarabel}.

Computation of the smoothing operator~\eqref{smoothing-operator} for nonnegative cones, second-order cones and positive semidefinite cones have analytical solutions, which have already discussed in the use of ABMM-IPM algorithm~\cite{ABIP}. We summarize the relevant computation in Appendix~\ref{appendix:smoothing-operator}.

Compared to nonnegative cones, second-order cones and positive-semidefinite cones, we can not obtain an analytical solution of~\eqref{smoothing-operator} for general conic constraints. Instead, we choose to compute~\eqref{smoothing-operator} numerically for general nonsymmetric cones via the damped Newton method, i.e.
\begin{align}
	s^{k+1} = s^k - \alpha_k [\nabla^2 f_{\mu}(s^k)]^{-1} \nabla f_{\mu}(s^k), \label{damped-Newton}
\end{align}
where 
\begin{align*}
	\alpha_k = \left\{\begin{array}{rl}
		\frac{1}{1+\lambda(f_{\mu}, s^k)} &, \quad \lambda(f_{\mu}, s^k) \ge \lambda^*\\
		1 &, \quad \text{else}
	\end{array}\right. ,
\end{align*}
$\lambda^* = 2-\sqrt{3}$ and $\lambda(f_{\mu}, s)$ is the Newton decrement defined as
$$
\lambda(f_{\mu}, s) := \norm{[\nabla^2 f_{\mu}(s)]^{-1/2}\nabla f_{\mu}(s)}, \ s \in \interior \mathcal{K}.
$$

We solve~\eqref{smoothing-operator} iteratively via~\eqref{damped-Newton} and it will finally converge to the unique solution of~\eqref{smoothing-operator}. We can show that the Newton method is computationally efficient on it for general nonsymmetric cones.
\begin{thm}
	Given a barrier function $f(s)$ that is self-concordant and lower-bounded, with the initial point $s^0 \in \interior \mathcal{K}$, $f_{\mu}(s)$ is also self-concordant and the damped Newton iterates are well defined, i.e. $s^k \in \interior\mathcal{K}$ and converge to the optimal solution of $f_{\mu}(s)$. Moreover, $s^k$ exhibits a quadratic convergence rate once the Newton decrement becomes sufficiently small, e.g. $\lambda(f_{\mu}, s^k) \le \lambda^*$.
\end{thm}
\begin{proof}
We can easily verify that the quadratic term $\frac{1}{2}\norm{s-c}^2$ satisfies~\eqref{concordant} and is self-concordant. Since the self-concordance is preserved under the direct summation (Proposition 2.1.1~\cite{Nesterov1994}), $f_{\mu}(s)$ remains to be self-concordant. 

Suppose $s^k \in \interior\mathcal{K}$. We define the step direction $\Delta^k = [\nabla^2 f_{\mu}(s^k)]^{-1} \nabla f_{\mu}(s^k)$, the Euclidean seminorm of $\Delta^k$ becomes
\begin{align*}
	\norm{s^{k+1} - s^k}_{s^k,f_{\mu}}^2 &= \frac{\norm{[\nabla^2 f_{\mu}(s^k)]^{-1}\nabla f_{\mu}(s^k)}_{s^k,f_{\mu}}^2}{(1+\lambda(f_{\mu}, s^k))^2} = \frac{\nabla f_{\mu}(s^k)^\top [\nabla^2 f_{\mu}(s^k)]^{-1}\nabla f_{\mu}(s^k)}{(1+\lambda(f_{\mu}, s^k))^2} \\
	& = \frac{\lambda(f_{\mu}, s^k)^2}{(1+\lambda(f_{\mu}, s^k))^2} < 1,
\end{align*}
i.e. the radius of the Dikin's ellipsoid at $s^{k}$ is less than $1$, which implies $s^{k+1} \in \interior \mathcal{K}$ (Theorem 2.1.1~\cite{Nesterov1994}). According to Theorem 2.2.1 in~\cite{Nesterov1994}, the damped Newton method is monotonically decreasing
\begin{align}
	f_{\mu}(s^{k+1}) \le f_{\mu}(s^k) - \omega(\lambda(f_{\mu}, s^k)), \label{monotone-decreasing}
\end{align}
where
\begin{align*}
	\omega(t) = t - \log(1+t)
\end{align*}
is monotonically increasing and positive for $\forall t > 0$. Since $f_{\mu}(s^k)$ is lower bounded and monotone decreasing, $\omega(\lambda(f_{\mu}, s^k))$ will converge to $0$ and the same for the Newton decrement $\lambda(f_{\mu}, s^k)$ by monotonicity. Once we have $\lambda(f_{\mu}, s^k) < {\lambda^*}$, it has been proved $\norm{s^{k+1} - s^k}_{s^k,f_{\mu}}^2 \le {\lambda^*}^2$ and the Newton decrement $\lambda(f_{\mu}, s^k)$ will converge quadratically under the standard Newton method with step size $\alpha_k=1$ (Theorem 2.2.3~\cite{Nesterov1994}) and the functional value $f_{\mu}(s^k)$ converges to the minimum of $f_{\mu}(s)$ (Theorem 2.2.2~\cite{Nesterov1994}).
\end{proof}

Since we usually exploit the warmstarting technique within a problem that needs to be solved repetitively with slight perturbation on parameters between two consecutive problems, it is likely that the initial point $s^0$ resides in the neighbourhood of the optimum of~\eqref{smoothing-operator} and the damped Newton method is prone to start from the region with quadratic convergence rate.

Note that two common nonsymmetric cones that are supported in modern conic solvers~\cite{MOSEK,Clarabel} are exponential and power cones. Both of them are $3$-dimensional and the Newton method will be computationally efficient for their smoothing operators~\eqref{smoothing-operator}.

\section{Primal-dual interior point methods}
The proposed warmstarting strategy is applicable for the class of primal-dual interior point methods. We test it on the Clarabel conic solver~\cite{Clarabel}, which is based on a primal-dual interior point method with a predictor-corrector framework. Since our warmstarting strategy is independent from other parts of a primal-dual interior point method, we omit the implementation of Clarabel and refer readers to~\cite{Clarabel} for details of Clarabel solver.
\subsection{Warmstarting with homogeneous embedding}
General conic solvers also need to consider the case when a problem is infeasible. The homogeneous embedding is widely used within an interior point method for infeasibility detection, e.g. Mosek~\cite{MOSEK} and Clarabel~\cite{Clarabel} solver. The central path is then defined slightly differently compared to~\eqref{general-central-path}, which is
\begin{align}\label{eqn:central_path_compact}
\begin{aligned}
	G(v) &= \mu G(v^0), \\
	z &= -\mu \nabla f(s),\\
	\tau \kappa & = \mu,
\end{aligned}
\end{align}
where the monotone mapping $G(v)$ is defined as, given $v := (x,z,s,\tau,\kappa)$,
\begin{equation}\label{eqn:root_finding_G}
	G(v) := 
	\begin{bmatrix}
		\hphantom{+}P  & A\tpose & q \\
		-A & 0 & b\\
		-q\tpose & -b\tpose & 0
	\end{bmatrix}
	\begin{bmatrix}
		x \\ z \\ \tau 
	\end{bmatrix}
	-
	\begin{bmatrix}
		0 \\ 0 \\
		\frac{1}{\tau}x\tpose P x
	\end{bmatrix}
	-
	\begin{bmatrix}
		0\\s\\ \kappa 
	\end{bmatrix}.
\end{equation}
The initial values $x^0,s^0,z^0$ are set as we proposed before. We then set the initial value $\tau^0 = 1$ and $\kappa^0 = \mu$ so that the initial pair $(\tau^0,\kappa^0)$ also lies on the new central path~\eqref{eqn:central_path_compact}.

\subsection{Termination criterion}
The interior point method is an iterative method and solves an optimization problem approximately to $\epsilon$-optimality. For checks of primal and dual feasibility, we define primal and dual objectives as 
\begin{align*}
	g_p &\eqdef \frac{1}{2}  x^\top P  x + q \tpose x  \\
	g_d &\eqdef -\frac{1}{2}  x^\top P  x - b\tpose  z,
\end{align*}
and use the primal residual $r_p$ and the dual residual $r_d$ defined in~\eqref{residual-map}.

We then declare convergence if all of the following three conditions holds:
\begin{align*}
	\norm{r_p} &< \epsilon \cdot \max\{1, \norm{b}_\infty + \norm{ x} + \norm{ s}\} \\
	\norm{r_d} &< \epsilon \cdot \max\{1, \norm{q}_\infty + \norm{ x} + \norm{ z}\} \\
	|g_p - g_d| &< \epsilon\cdot  \max\{1, \min\{|g_p|, |g_d|\}\}.
\end{align*}

The precision is set to $\epsilon = 10^{-8}$ by default for later experiments if we do not specify it explicitly. For infeasibility check, we declare primal infeasibility if 
\begin{align*}
	\norm{A\tpose z} & < -\epsilon_{i,r} \cdot \max(1, \norm{x} + \norm{z}) \cdot (b\tpose z) \\
	b\tpose z & < -\epsilon_{i,a},
\end{align*}
and dual infeasibility if 
\begin{align*}
	\norm{Px} & < -\epsilon_{i,r} \cdot \max(1, \norm{x}) \cdot (b\tpose z) \\
	\norm{Ax + s} & < -\epsilon_{i,r} \cdot \max(1, \norm{x} + \norm{s}) \cdot (q\tpose x) \\
	q\tpose x  & < -\epsilon_{i,a},
\end{align*}
where we set $\epsilon_{i,r} = \epsilon_{i,a} = 10^{-8}$ by default.

\section{Numerical Experiments}
In this section, we evaluate the efficacy of the proposed smoothing operator~\eqref{smoothing-operator} for warmstarting across a variety of problem sets. \footnote{See tests on \url{https://github.com/oxfordcontrol/Clarabel.jl/tree/yc/warmstart/warmstart_test}.} We begin by outlining the general methodology employed in our testing. Subsequently, we present the results, focusing on applications such as hyperparameter tuning in machine learning and iterative reoptimization in portfolio optimization under diverse models. Lastly, we check how the effectiveness of our warmstarting is affected by the magnitude of perturbation on parameters.  

\subsection{General metric}
We compare the proposed warmstarting method with the standard coldstarting used in current solvers~\cite{MOSEK,Clarabel}. Both are implemented based on the Clarabel solver~\cite{Clarabel}.
We use the same metric in~\cite{Skajaa2013} to measure the efficacy of our warmstarting technique. Given a problem $\mathcal{P}_i$, we define the reduction ratio of the warmstarting technique as
\begin{align*}
	\mathcal{R}_i = \frac{\text{number of iterations of the warmstarted IPM on }\mathcal{P}_i}{\text{number of iterations of the coldstarted IPM on }\mathcal{P}_i}.
\end{align*}

For a set of test problems $\mathcal{R}_1,\dots,\mathcal{R}_N$, we evaluate the overall reduction rate $\mathcal{R}$ of warmstarting by the geometric mean defined as
\begin{align*}
	\mathcal{R} = \left(\prod_{i=1}^N \mathcal{R}_i \right)^{\frac{1}{N}}.
\end{align*}
A lower value of $\mathcal{R}$ means the warmstarting is more effective on reducing computation over test problems.

\subsection{Tuning hyperparameters in support vector machines}
The support vector machine (SVM) is a classical machine learning model where the training process can be formulated as solving a convex optimization problem. Suppose we have a set of training samples $\left\{x_i, y_i\right\}_{i=1}^m \subseteq \mathbb{R}^n \times\{-1,+1\}$ for a standard binary classification problem, the final classifier will be $h_{w, b}(x)=\operatorname{sgn}(\langle w, x\rangle+b)$, where the parameters $w, b$ are the solution of the following convex optimization problem:
\begin{align}
\begin{aligned}
	\min_{w, b, \xi}: \quad& \frac{1}{m}\sum_{i=1}^m \xi_i + R(w, b) \\
	\text{ s.t.} \quad& \xi_i \geq 1-y_i\left(\left\langle w, x_i\right\rangle+b\right), \ \forall i = 1, \dots, m \\
	\quad & \xi_i \geq 0,
\end{aligned}\label{svm-original}
\end{align}
where $\zeta \in \mathbb{R}^m$ is the slack variable characterizing prediction error and $R(w,b)$ is the regularization for parameters $w,b$.

\subsubsection{SVM with $L_1$ regularization}
When we choose the norm-1 regularization, i.e. $R(w,b) = \lambda \norm{w}_1$, the problem~\eqref{svm-original} becomes a linear program where $\lambda > 0$ is a hyperparameter that should be finely tuned. We initially solve~\eqref{svm-original} with $\lambda=0.01$ and then obtain a pair of parameter $(w^*,b^*)$. Then, we increase $\lambda$ by $0.01$ and warmstart the new problem with $(w^*,b^*)$ and repeat the tuning process several times for better predicting performance, which is called the hyperparameter tuning in machine learning.

We take half of data from the MNIST dataset and train a binary SVM classifier with the norm-$1$ regularization. The results for both warmstarting and coldstarting are shown in Table~\ref{table:svm_hyperparameter_tuning}. The overall reduction rate of the iteration number is $\mathcal{R}_{iter} = 0.4984$ and the reduction rate of the solve time is $\mathcal{R}_t = 0.4561$, which shows our warmstarting technique is effective for the hyperparameter tuning. 

\begin{longtable}{||c||cc||cc||}
	\caption{Tuning the regularization parameter $\lambda$ in SVMs}
	\label{table:svm_hyperparameter_tuning}
	\\
	& \multicolumn{2}{c||}{iterations}& \multicolumn{2}{c||}{solve time (s)}\\[2ex] 
	$ \lambda $ value & Warm & Cold & Warm & Cold\\[1ex]
	\hline
	\endhead
	\sc{0.02} & 33 & 49 & 47.4 & 72.9\\ 
	\sc{0.03} & 21 & 20 & 30.7 & 31.2\\ 
	\sc{0.04} & 11 & 20 & 15.6 & 30.7\\ 
	\sc{0.05} & 10 & 19 & 14.1 &   29\\ 
	\sc{0.06} & 9 & 20 & 12.8 &   32\\ 
	\sc{0.07} & 8 & 21 & 11.5 & 33.4\\ 
	\sc{0.08} & 8 & 20 & 11.4 &   30\\ 
	\sc{0.09} & 8 & 19 & 10.8 &   28\\ 
	\sc{ 0.1} & 8 & 20 & 10.9 & 31.8\\ 
	\sc{0.11} & 8 & 20 & 11.1 & 31.6\\ 
\end{longtable}

\subsubsection{Robust SVM with $L_2$ regularization}
When we choose the norm-2 regularization, i.e. $R(w,b) = \lambda \norm{w}_2$, the problem~\eqref{svm-original} can be interpreted as a robust optimization problem~\cite{Xu2009}, which can be reformulated as a second-order cone program. We initialize $\lambda = 0.01$ and then repeat the tuning of $\lambda$ with increment equal to $0.01$.
\begin{longtable}{||c||cc||cc||}
	\caption{Tuning the regularization parameter $\lambda$ in robust SVMs}
	\label{table:svm_robust_hyperparameter_tuning}
	\\
	& \multicolumn{2}{c||}{iterations}& \multicolumn{2}{c||}{solve time (s)}\\[2ex] 
	$ \lambda $ value & Warm & Cold & Warm & Cold\\[1ex]
	\hline
	\endhead
	\sc{0.02} & 30 & 58 &   43 & 84.6\\ 
	\sc{0.03} & 28 & 53 & 40.2 & 77.3\\ 
	\sc{0.04} & 27 & 50 &   39 & 72.5\\ 
	\sc{0.05} & 28 & 50 & 40.1 & 72.8\\ 
	\sc{0.06} & 29 & 53 & 41.5 & 77.8\\ 
	\sc{0.07} & 31 & 48 & 43.2 & 67.9\\ 
	\sc{0.08} & 28 & 45 & 38.7 & 67.2\\ 
	\sc{0.09} & 30 & 45 & 42.9 & 65.9\\ 
	\sc{ 0.1} & 28 & 43 &   40 &   63\\ 
	\sc{0.11} & 28 & 41 & 40.2 & 61.7\\ 
\end{longtable}
The overall reduction rate of the iteration number is $\mathcal{R}_{iter} = 0.5931$ and the reduction rate of the solve time is $\mathcal{R}_t = 0.5775$.

\subsection{Portfolio optimization}
Portfolio optimization is a classical model used in quantitative finance~\cite{cornuejols2006}, which is of the following form:
\begin{align}
	\begin{aligned}
		\min _{x} \quad \mathcal{R}\left(-r^{\top} x\right) \\
		\text { s.t. } \quad \mathbf{e}^{\top} x =1, \\
		\bar{r}^{\top} x \geq r_0, \\
		\quad x \geq 0.
	\end{aligned}\label{portfolio-original}
\end{align}
The return $r \in \mathbb{R}^n$ for each asset is unknown and we need to estimate the expected return $\bar{r}$ from historical data. The problem~\eqref{portfolio-original} aims to minimize the risk of loss $\mathcal{R}(\cdot)$ if we aim to find a portfolio $x \in \mathbb{R}^n$ that the expected return is no less than $r_0$. The sum of portfolio is normalized to $1$ and $x \ge 0$ imposes no short-selling constraint.

Suppose we consider the standard mean-variance model~\cite{cornuejols2006} where we choose the risk $\mathcal{R}$ to be the variance of random variable $r$. The problem~\eqref{portfolio-original} can be reformulated as a second-order cone program by introducing the slack variable $t$:
\begin{align}
	\begin{aligned}
		\min _{t,x} \quad & \quad t \\
		\text { s.t. } \quad & \mathbf{e}^{\top} x =1, \\
		& \sqrt{x^\top \Sigma_r x} \le t, \\
		&\bar{r}^{\top} x \geq r_0, \\
		& x \geq 0,
	\end{aligned}\label{portfolio-efficient-frontier}
\end{align}
where $\Sigma_r$ is the variance matrix for the return variable $r$. Given the return matrix $R \in \mathbb{R}^{d \times n}$ for the last $d$ days, the expected return $\bar{r}$ is the mean of $R$ over each column and the variance matrix $\Sigma_r$ is obtained from $\Sigma_r = (d-1)^{-1}(R-\mathbf{e}\bar{r}^\top)^\top(R-\mathbf{e}\bar{r}^\top)$. We input the variance constraint as an second-order constraint $(t,Ux) \in \mathcal{K}_{\mathrm{soc}}^n$ in Clarabel solver~\cite{Clarabel}, where $U$ is from the Cholesky factorization of the variance matrix $\Sigma_r = U^\top U$. We set $d = 500$ and $n = 300$ in this section and the historical return $R$ is extracted from $n$ stocks in the S\&P500 stock index. We first test frequent portfolio rebalancing and efficient frontier problems similar to~\cite{Skajaa2013}, and then test with a different risk measure based on the power cone~\cite{Krokhmal2007}.

\subsubsection{Frequent portfolio rebalancing}\label{subsec:portfolio-rebalance}
Given a fixed selection of $n$ assets, the portfolio optimization problem~\eqref{portfolio-efficient-frontier} needs to be reoptimized over time with slight changes for values of $\bar{r},\Sigma_r$ as they are estimated from the last $d$ days. The problem can be regarded as a parametric second-order cone programming w.r.t. $\bar{r}, \Sigma_r$. We estimate parameters $\bar{r},\Sigma_r$ by the same procedure everyday, where the time horizon is shifted by one day compared to the last estimate, and reoptimize the portfolio with new estimates $\bar{r},\Sigma_r$ in~\eqref{portfolio-efficient-frontier}. We expect the optimal solution is close to the solution from the last day and the warmstarting should be effective in reducing computational time. We reuse the optimal solution $(t^*,x^*)$ from the last day as the input for the smoothing operator~\eqref{smoothing-operator}. The output of the smoothing operator is then used as the warmstarting point for the next run.

\begin{figure}
	\centering
	\caption{100 consecutive transactions for the mean-variance model}
	\includegraphics[width=0.8\textwidth]{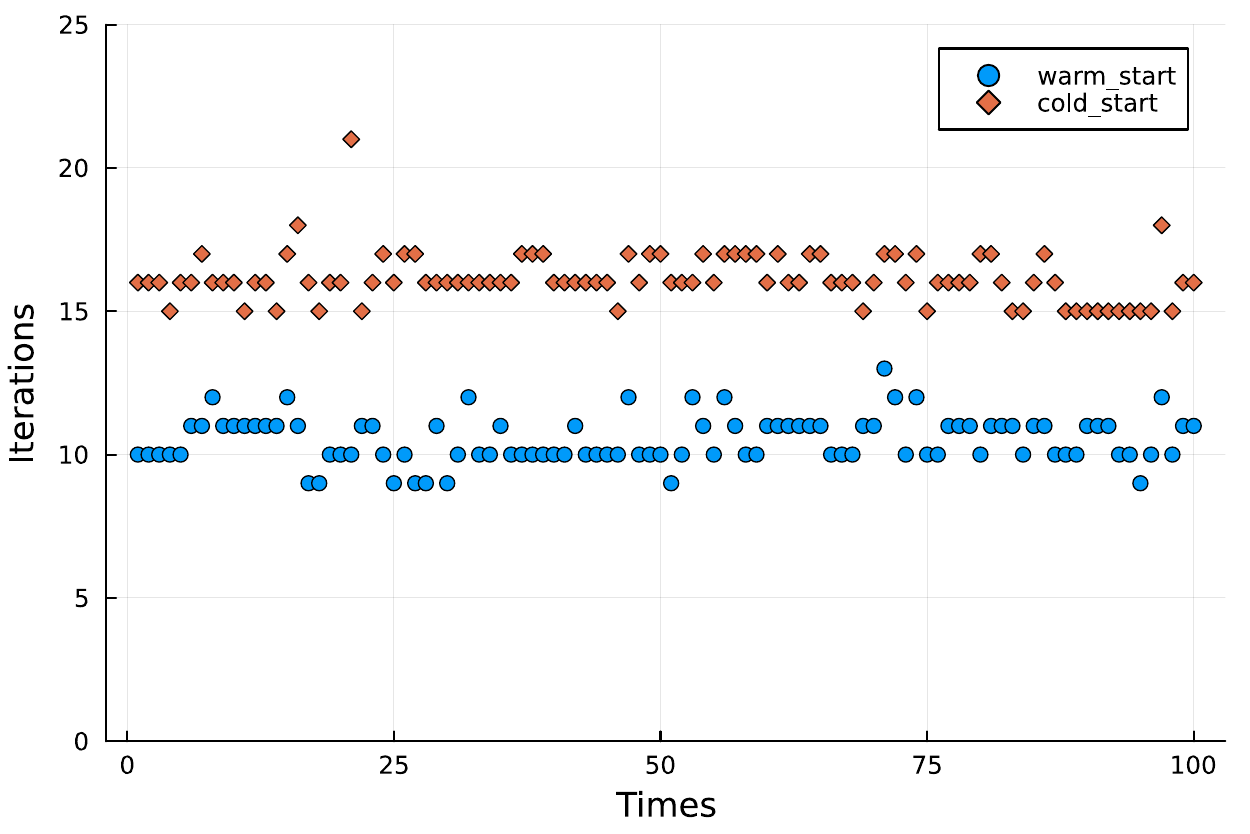}
	\label{fig:portfolio-daily-transaction}
\end{figure}

The results are shown in Figure~\ref{fig:portfolio-daily-transaction} where the transaction is simulated for $N=100$ repeated times. We find that the warmstarting is always faster than the coldstarting in the portfolio rebalancing problem, and the overall reduction rate of the iteration number is $\mathcal{R}_{iter} = 0.6277$ and the reduction rate of the solve time is $\mathcal{R}_t = 0.5934$.

\subsubsection{Efficient frontier}
Computing efficient frontier in the Markowitz portfolio selection can also benefit from warmstarting strategies when we need to solve a sequence of portfolio optimization problems with different choices of the minimum return $r_0$ in~\eqref{portfolio-efficient-frontier}. The \emph{efficient frontier} of problem~\eqref{portfolio-efficient-frontier} is the pair of points $(r_0, f (r_0))$ for $t \in [0,\max(\bar{r})]$ where $f(r_0)$ is the optimum of~\eqref{portfolio-efficient-frontier} that is parametrized by the selection of $r_0$. We compute the optimal solution of problem~\eqref{portfolio-efficient-frontier} with the initial value $r_0 = 0.001$ and then increase $r_0$ by $0.0001$ for every new problem up to $r_0 = 0.002$, where we apply the warmstarting technique we proposed and compare it with the coldstart.

The results are shown in Table~\ref{table:efficient-frontier} where we list values of $r_0$ and $f(r_0)$, along with the corresponding iteration number and solve time. The risk $f(r_0)$ is increasing as we set higher return goal for our portfolio, which is consistent with our expectation. The overall reduction rate of the iteration number is $\mathcal{R}_{iter} = 0.5353$, which is close to the reduction rate of the solve time is $\mathcal{R}_t = 0.5021$. It implies the additional time spent on warmstarting is negligible compared to solve time within an interior point method for second-order cone optimization.
\begin{longtable}{||cc||cc||cc||}
	\caption{Efficient frontier}
	\label{table:efficient-frontier}
	\\
	\multicolumn{2}{||c||}{{Values}} & \multicolumn{2}{c||}{{iterations}}& \multicolumn{2}{c||}{{solve time (s)}}\\[2ex] 
	\hline
	$r_0 $ & $f(r_0)$ & Warm & Cold & Warm & Cold\\[1ex]
	\hline
	\endhead
	\sc{0.0011} & \sc{0.0137} & 11 & 20 & 0.0868 & 0.172\\ 
	\sc{0.0012} & \sc{0.0138} & 10 & 20 & 0.0811 & 0.171\\ 
	\sc{0.0013} & \sc{0.014} & 9 & 19 & 0.0721 & 0.161\\ 
	\sc{0.0014} & \sc{0.0142} & 9 & 19 & 0.072 & 0.156\\ 
	\sc{0.0015} & \sc{0.0145} & 12 & 22 & 0.0929 & 0.181\\ 
	\sc{0.0016} & \sc{0.0147} & 11 & 18 & 0.0893 & 0.15\\ 
	\sc{0.0017} & \sc{0.015} & 13 & 19 & 0.101 & 0.16\\ 
	\sc{0.0018} & \sc{0.0153} & 11 & 20 & 0.0851 & 0.172\\ 
	\sc{0.0019} & \sc{0.0157} & 10 & 21 & 0.0769 & 0.184\\ 
	\sc{0.002} & \sc{0.016} & 11 & 21 & 0.127 & 0.179\\ 
\end{longtable}
\subsubsection{Higher-moment coherent risk measures}
Higher-moment coherent risk measures (HMCR) is the generalization of the conditional value-at-risk measure (CVaR),
$$ HMCR_{p,\alpha}(X) := \min_{\eta} \eta + (1-\alpha)^{-1}\norm{(X-\eta)^+}_p, p > 1,$$
and works effectively within the stochastic programming of problem~\eqref{portfolio-original}. Suppose we have the return matrix $R \in \mathbb{R}^{d \times n}$ recording returns of $n$ assets across $d$ consecutive days. we choose the risk measure to be the higher-moment coherent risk measures, $\mathcal{R}(X) = HMCR_{p,\alpha}(X)$, and take $d$ different days as the scenarios in stochastic programming. Then the problem~\eqref{portfolio-original} is equivalent to the following form:
\begin{align}
\begin{aligned}
	\min \quad & \eta + \frac{t}{(1-\alpha)d^{\frac{1}{p}}}\\
	s.t.  \quad & \mathbf{e}^\top x = 1, \\
		& \frac{1}{d}\sum_{j=1}^{d}\sum_{i=1}^n R_{j,i}x_i \ge r_0,\\
		& w \ge -Rx - \eta \cdot \mathbf{e}, \\
		& t \ge \left(w_1^p + \dots + w_d^p\right)^{\frac{1}{p}},\\
		& x \ge 0, w \ge 0,
\end{aligned}
\end{align}
where $w_i, i=1,\dots,d$ is the slack variable for $i$-th scenario. The expected return over $d$ scenarios should be no less than $r_0$. The $p$-th moment coherent risk measures can be formulated as a $p$-norm conic constraint, which is equivalent to $d$ power cone constraints plus a linear constraint as follows~\cite{MOSEK}:
\begin{align*}
\begin{aligned}
	\left(r_i, t, w_i\right) & \in \mathcal{K}_{\text{pow}}^{1 / p}, \ \forall i = 1, \dots, d, \\
	\sum_{i=1}^d r_i & =t.
\end{aligned}
\end{align*}

Results are shown in Figure~\ref{fig:portfolio-higher-moment-coherent-risk-measures}. We set the accuracy level to $\epsilon = 1e^{-7}$ and simulate it for $N=100$ times. The overall reduction rate for the iteration number of and solve time are $\mathcal{R}_{iter} = 0.6191$ and $\mathcal{R}_{t} = 0.6013$ respectively. $\mathcal{R}_{iter} \approx \mathcal{R}_{t}$ shows that the computational time for the smoothing operator~\eqref{smoothing-operator} is negligible within the warmstarting scheme for this portfolio optimization problem with power cones.
\begin{figure}
	\centering
	\caption{100 consecutive transactions with higher-moment coherent risk measures}
	\includegraphics[width=0.8\textwidth]{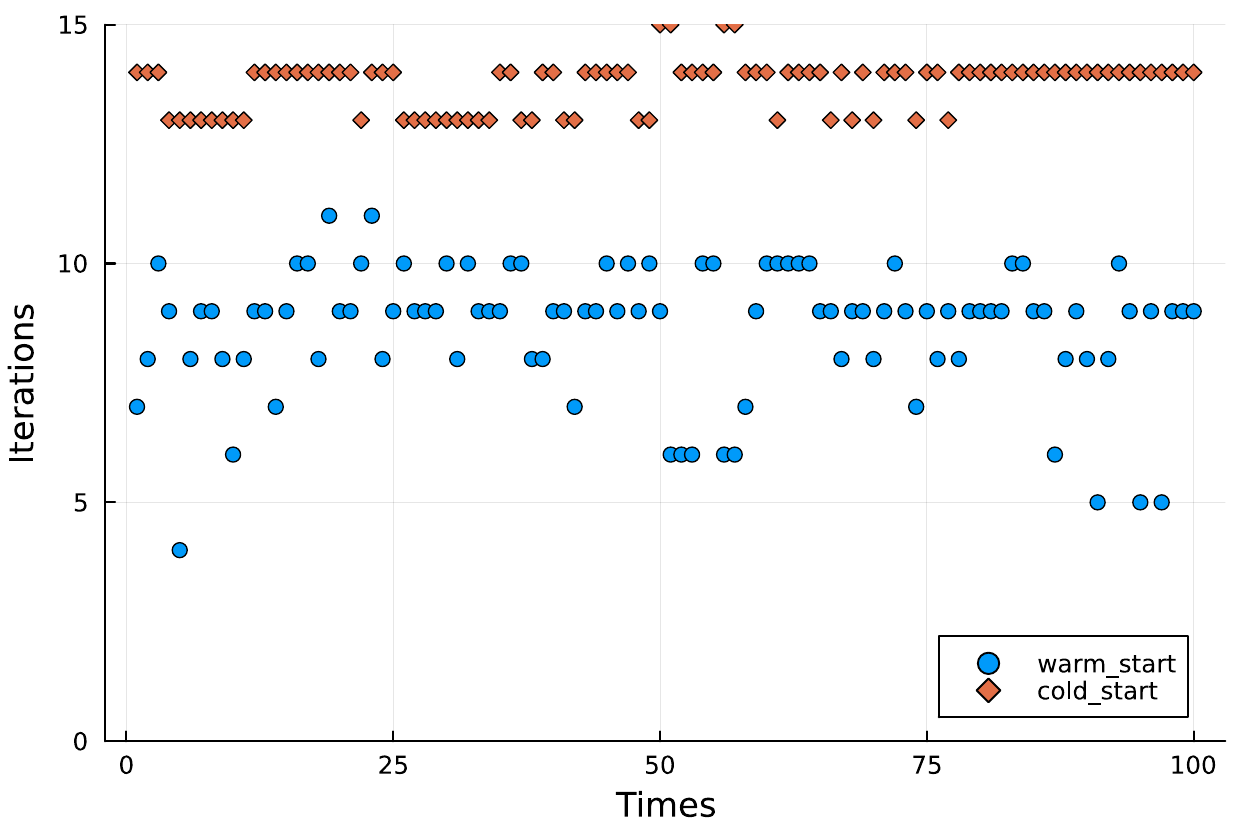}
	\label{fig:portfolio-higher-moment-coherent-risk-measures}
\end{figure}

\subsection{Perturbation effects on the warmstarting strategy}
Finally, we test how the performance of the warmstarting strategy varies with changes in the magnitude of the perturbation. We benchmark the model predictive control (MPC) problems~\cite{Borrelli:MPCbook} with quadratic objectives from the benchmark collection~ \cite{Kouzoupis:2015}, which are in the form
\begin{align}
\begin{aligned}
	\min_{y,x,u} \quad & \sum\limits_{i=0}^{N-1}
	\begin{pmatrix}
		y_i - y^r_i \\
		u_i - u^r_i
	\end{pmatrix}
	\begin{pmatrix}
		Q_k & S_k \\ S_k^T & R_k
	\end{pmatrix}
	\begin{pmatrix}
		y_i - y^r_i\\
		u_i - u^r_i
	\end{pmatrix} \hfill 
	+   \begin{pmatrix}
		g^y_k \\ g_u^k 
	\end{pmatrix}\tpose 
	\begin{pmatrix}
		y_i - y^r_i\\
		u_i - u^r_i
	\end{pmatrix} \\
	& + (x_N - x^r_N)\tpose P(x_N - x^r_N) \\
	\text{s.t} \quad &     \left.
	\begin{aligned}
		&x_{k+1} = A_k x_k + B_ku_k + f_k \\[0.5ex]
		&y_{k}   = C_k x_k + D_ku_k + e_k \\[0.5ex]
		&d_k^\ell \le M_k x_k + N_ku_k \le d_k^u \\[0.5ex]
		&u_k \in \mathcal{U}_k,~y_k \in \mathcal{Y}_k \\[0.5ex]
	\end{aligned}  \quad \right\} \quad k = 0 \dots N-1 \\
	&Tx_N \in \mathcal{\mathcal{T}},
\end{aligned}
\end{align}

where the constraint sets $\mathcal{U}_k$, $\mathcal{Y}_k$ and $\mathcal{T}$ are interval constraints and cost matrices satisfy ${Q}_k \succeq 0$, ${R}_k \succeq 0$ and $P \succ 0$.  The dimension of the states $x_k$ and inputs $u_k$ are relatively small (max 12 and 4, respectively), with horizons $N$ up to 100.  

Similar to~\cite{Skajaa2013}, we generate the perturbation $\delta$ for parameters $b,q,A$ respectively, where $10 \%$ but at most 20 elements are changed after perturbation. Given the changing parameter $v$, each entry of $v$ is modified as follows,
$$
v_i:= \begin{cases}\delta r & \text { if }\left|v_i\right| \leq 10^{-6} \\ (1+\delta r) v_i & \text { otherwise }\end{cases}
$$
where $r$ is a random number generated from the uniform distribution $[-1,1]$. Figure~\ref{fig:mpc_varying_disturbance} shows how the size of perturbation affects the reduction ratio of the proposed warmstarting strategy. The y-axis denotes the geometric mean of the number of iterations for MPC problems. The warmstarting is quite effective when the perturbation $\delta$ is small enough and the geometric mean of the iteration number is non-decreasing as the magnitude of perturbation $\delta$ grows up.  
\begin{figure}
	\centering
	\caption{Geometric mean $\mathcal{R}$ vs perturbation $\delta$}
	\includegraphics[width=0.8\textwidth]{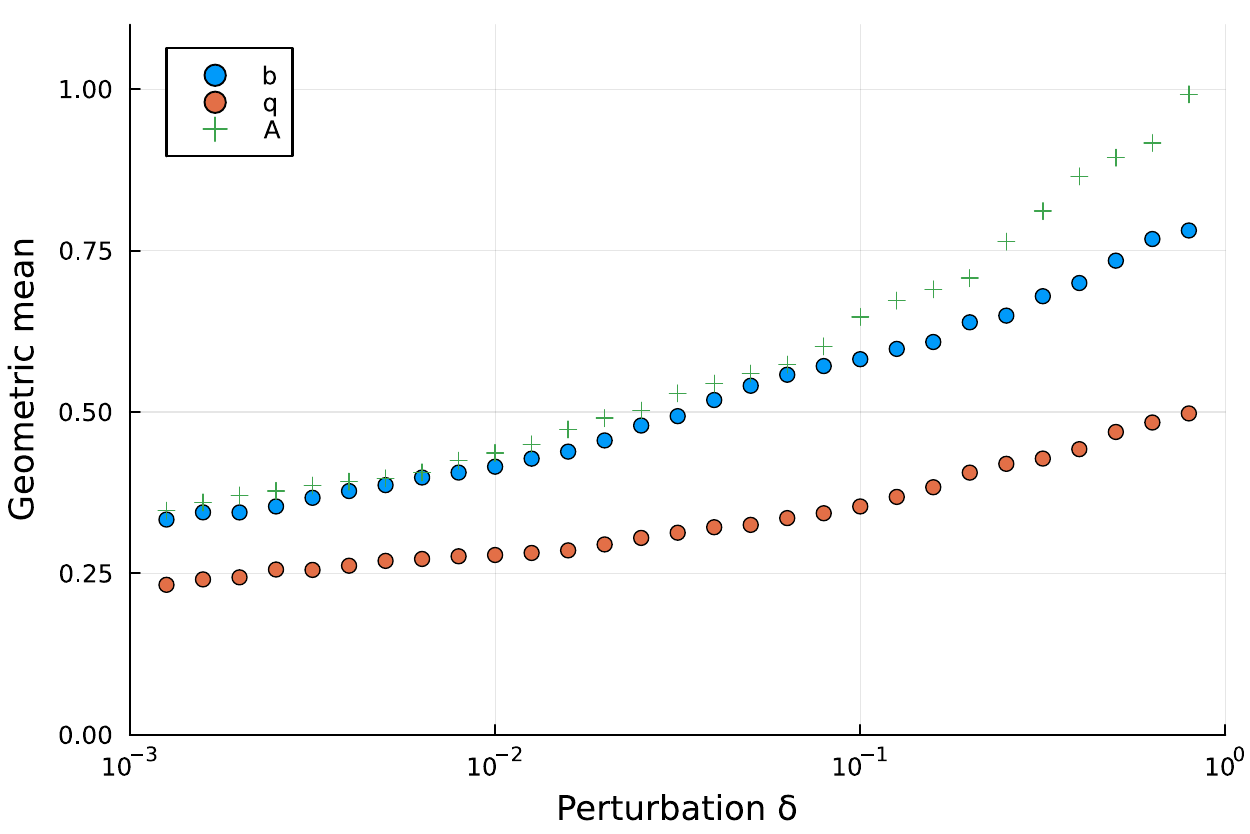}
	\label{fig:mpc_varying_disturbance}
\end{figure}

\section{Conclusion}
In this paper, we have proposed a warmstarting method that is applicable for conic optimization problems. Our warmstarting is based on a smoothing operator that can be computed in parallel for different cones. We have shown that the smoothing operator has an analytic solution for nonnegative cones, second-order cones and positive semidefinite cones, and the Newton method can find the unique solution of the smoothing operator with a locally quadratic convergence rate for any cone of a self-concordant barrier function. Compared to previous work on warmstarting, our method can generate an initial point on the central path of a primal-dual interior point method and the analysis shows that the infeasibility residuals of our initial point is at most $O(\mu)$ deviation from that of the optimum of the last problem. 

Note that nonnegative cones, second-order cones and positive semidefinite cones all belong to the class of symmetric cones. Future work can be extending the current $O(\mu)$ error analysis for general symmetric cones, or nonsymmetric cones if possible.

\clearpage

\bibliographystyle{unsrt}
\bibliography{bibfile}

\appendix

\section{Barrier functions for a class of cones}\label{appendix:barrier-function}
We list barrier functions for cones that are commonly supported in state-of-the-art conic optimization solvers~\cite{MOSEK,Clarabel,CuClarabel}:
\begin{itemize}
	\item
	The \emph{nonnegative cone}, which is used to formulate linear inequality constraints, is defined as
	$$
	\Re_{+}^n:=\left\{x \in \Re^n \ \middle| \ x_i \geq 0, \ \forall i=1, \ldots, n\right\},
	$$
	with its LHSCB function of degree $n$,
	$$
	f(x) = -\sum_{i \in \llbracket n \rrbracket}\log(x_i), \ x \in \interior(\mathcal{K}_{\ge}^n).
	$$
	\item 
	The \emph{second-order cone} $\mathcal{K}_{\text {soc }}^n$ (also sometimes called the \emph{quadratic} or \emph{Lorentz cone}), is defined as
	$$
	\mathcal{K}_{\mathrm{soc}}^n:=\left\{(t, x) \ \middle| \ x \in \Re^{n-1}, t \in \Re_{+},\|x\|_2 \leq t\right\},
	$$
	with the LHSCB function of degree $1$, 
	$$
	f(x) = -\frac{1}{2}\log\left(x_1^2 - \sum_{i =2}^n x_i^2\right), \ x \in \interior(\mathcal{K}_{q}^n).
	$$
	\item The \textit{positive semidefinite cone} $\mathcal{K}_{\succeq}^n$ is defined as:
	$$
	\mathcal{K}_{\succeq}^n := \set{x \in \mathbb{R}^{n(n+1)/2}}{\textrm{mat}(x) \succeq 0},
	$$
	with the LHSCB function of degree $n$,
	$$
	f(x) = - \log \det \left(\mat(x)\right), \ \mat(x) \in \interior(\mathcal{K}_{\succeq}^n).
	$$
	\item
	The \emph{exponential cone} is a $3$-dimensional cone defined as
	$$
	\mathcal{K}_{\exp }:=\left\{x \in \mathbb{R}^3 \ \middle| \ x_2\ge 0, x_2 \exp \left(\frac{x_3}{x_2}\right) \leq x_1\right\} \cup\{(x_1, 0, x_3) \mid x_1 \geq 0,x_3 \leq 0 \},
	$$
	with the LHSCB function of degree $3$,
	$$
	f(x)=-\log \left(x_2 \log \left(x_1 / x_2\right)-x_3\right)-\log x_1-\log x_2,\ x \in \interior (\mathcal{K}_{\mathrm{exp}}).
	$$
	\item
	The $3$-dimensional \emph{power cone} with exponent $\alpha \in (0,1)$ is defined as
	$$
	\mathcal{K}_{\text{pow},\alpha} = \left\{x \in \mathbb{R}^3 \ \middle| \ x_1^\alpha x_2^{1-\alpha} \geq \vert x_3 \vert, x_1,x_2 \geq 0\right\},
	$$
	with the LHSCB function of degree $3$,
	\begin{small}
		\begin{equation}
			f(x) = -\log\left({x_1}^{2\alpha} {x_2}^{2(1-\alpha)} - x_3^2\right) - (1-\alpha)\log(x_1) - \alpha\log(x_2), \ x \in \interior (\mathcal{K}_{\mathrm{pow}}).\label{pow-barrier}
		\end{equation}
	\end{small}
\end{itemize}

\section{Computation for smoothing operators over cones}\label{appendix:smoothing-operator}
\subsubsection*{Nonnegative cones}
For nonnegative cones $\mathcal{K} = \mathbb{R}^n_{\ge 0}$ that generalize linear inequality constraints, the LHSCB function is $f(s) = -\sum_{i=1}^{n}\log(s_i)$ and the 1st-order optimality condition of the inner problem of~\eqref{smoothing-operator} is
\begin{align*}
	s_i - c_i - \frac{\mu}{s_i} = 0, \ \forall i \in \indexset{n},
\end{align*}
which yields
\begin{align}
	s_i = \frac{c_i + \sqrt{c_i^2 + 4 \mu}}{2}
\end{align}
since $s_i > 0, \forall i \in \indexset{n}$. 
\subsubsection*{Second-order cones}\label{subsec-socp}
The LHSCB function $f(\cdot)$ for a second-order cone is
\begin{align*}
	f(s) = -\frac{1}{2}\log\left(s_0^2 - \norm{s_1}^2\right), \ \forall (s_0,s_1) \in \interior\mathcal{K}_{\text{soc}}^n.
\end{align*}
Suppose we denote $t_s:= s_0^2 - \norm{s_1}^2$, the 1st-order optimality condition of~\eqref{smoothing-operator} is
\begin{align}
\begin{aligned}
	-\frac{\mu}{t_s}\begin{bmatrix}
		s_0 \\ -s_1
	\end{bmatrix}
	+ \begin{bmatrix}
		s_0 \\ s_1
	\end{bmatrix}
	- \begin{bmatrix}
		c_0 \\ c_1
	\end{bmatrix} = 0.
\end{aligned}\label{1st-order-optimality-soc}
\end{align}
We compute $(s_0,s_1)$ in two different ways depending on the value of $c_0$. 
\begin{itemize}
	\item If $c_0 = 0$, the first equality of~\eqref{1st-order-optimality-soc} yields $\mu = t_s = s_0^2 - \norm{s_1}^2$ and the second equality becomes $s_1 = c_1/2$. Therefore, we can compute $s_0$ via 
	$$
	s_0 = \sqrt{\mu + \norm{s_1}^2} = \sqrt{\mu + \norm{c_1}^2/4}.
	$$
	
	\item When $c_0 \neq 0$, the optimality condition~\eqref{1st-order-optimality-soc} implies
	\begin{align*}
		s_0 = \frac{t_s}{t_s-\mu}c_0, \quad s_1 = \frac{t_s}{t_s+\mu}c_1,
	\end{align*}
	and we have the equation
	\begin{align*}
		t_s = \left(\frac{t_s}{t_s-\mu}\right)^2c_0^2 - \left(\frac{t_s}{t_s+\mu}\right)^2\norm{c_1}^2,
	\end{align*}
	which reduces to
	\begin{align*}
		(t_s^2 - \mu^2)^2 = t_s(t_s+\mu)^2 c_0^2 - t_s(t_s-\mu)^2\norm{c_1}^2.
	\end{align*}
	We define $\rho = t_s/\mu$ and the equation above becomes
	\begin{align*}
		(\rho^2-1)^2 = \frac{\rho}{\mu}(\rho+1)^2c_0^2 - \frac{\rho}{\mu}(\rho-1)^2\norm{c_1}^2,\\
		\rho^4 - 2\rho^2+1 = \frac{c_0^2-\norm{c_1}^2}{\mu}\rho^3 + \frac{2(c_0^2+\norm{c_1}^2)}{\mu}\rho^2 + \frac{c_0^2-\norm{c_1}^2}{\mu}\rho,\\
		\rho^2 + \frac{1}{\rho^2} - \frac{c_0^2-\norm{c_1}^2}{\mu}\left(\rho + \frac{1}{\rho}\right) - \frac{2(c_0^2+\norm{c_1}^2)}{\mu} - 2 = 0,
	\end{align*}
	which reduces to
	\begin{align}
		\gamma^2 - \frac{c_0^2-\norm{c_1}^2}{\mu}\gamma - \frac{2(c_0^2+\norm{c_1}^2)}{\mu} - 4 = 0
	\end{align}
	where we define $\gamma := \rho + 1/\rho > 2$, and we can obtain
	\begin{align*}
		\gamma = \frac{\frac{c_0^2-\norm{c_1}^2}{\mu} + \sqrt{\left(\frac{c_0^2-\norm{c_1}^2}{\mu}\right)^2 + \frac{8(c_0^2+\norm{c_1}^2)}{\mu}+16}}{2}.
	\end{align*}
	Then, we can obtain the value of $\rho$ by
	\begin{align*}
		\rho = \left\{
		\begin{matrix}
			\frac{\gamma+\sqrt{\gamma^2-4}}{2}, & c_0 > 0,\\
			\frac{\gamma -\sqrt{\gamma^2-4}}{2}, & c_0 < 0,
		\end{matrix}
		\right..
	\end{align*}
	$s_0,s_1$ can be computed via
	\begin{align*}
		s_0 = \frac{\rho}{\rho-1}c_0, \quad s_1 = \frac{\rho}{\rho+1}c_1.
	\end{align*} 
\end{itemize}
\subsubsection*{Positive semidefinite cones}
For positive semidefinite (PSD) cones where the LHSCB function is
\begin{align*}
	f(S) = - \log \det (S), \ \forall S \in \mathbb{S}^n_+,
\end{align*}
where $S = \mat(s)$. The optimality condition of~\eqref{smoothing-operator} becomes, given $C = \mat(c)$,
\begin{align*}
	S - C - \mu (S)^{-1} = 0.
\end{align*}
Following the same derivation for the proximal operator of semidefinite cones from~\cite{ABIP}, we can obtain $S$ by $S = Q^\top E Q$, where $Q$ is from the eigenvalue decomposition of $C = Q^\top D Q$ and $E$ is a diagonal matrix that
\begin{align}
	E = \diag(e), \ e_i = \frac{d_i + \sqrt{d_i^2 + 4\mu}}{2}, \forall i \in \indexset{n}, \label{psd-solution}
\end{align}
where $d_i$ is the eigenvalue of $C$ in the diagonal matrix $D = \diag(d)$.
\end{document}